\documentclass[a4paper]{article}
\usepackage{amstext,amssymb,amsmath,amsbsy}

\usepackage[left=2.65cm,right=2.65cm,top=3.2cm,bottom=3.2cm]{geometry}

\usepackage[numbers]{natbib}

\usepackage{tikz}
\usepackage{hyperref}
\usepackage{amscd}
\usepackage{amsfonts}
\usepackage{indentfirst}
\usepackage{verbatim}
\usepackage{amsmath}
\usepackage{amsthm}
\usepackage{enumerate}
\usepackage{graphicx}
\usepackage{subfig}
\usepackage{cleveref}

\usepackage{xcolor}

\usepackage[OT1]{fontenc}
\usepackage[latin1]{inputenc}
\usepackage[english]{babel}
\usepackage{amssymb}
\newtheorem{theorem}{Theorem}
\newtheorem{lemma}{Lemma}

\newtheorem{remark}{Remark}

\newcommand{\cF}{{\mathcal F}}
\newcommand{\cB}{{\mathcal B}}
\newcommand{\cD}{{\mathcal D}}
\newcommand{\hcD}{\hat {\mathcal D}}

\newcommand{\tw}{\widetilde w}

\newcommand{\hw}{\hat w}

\newcommand{\tr}{^\mathsf{T}}

\numberwithin{equation}{section}

\newcommand{\dsp}{\displaystyle}
\newcommand{\cA}{{\mathcal A}}
\newcommand{\hv}{\hat v}
\newcommand{\eps}{\varepsilon}
\newcommand{\mN}{\mathbb{N}}
\newcommand{\mR}{\mathbb{R}}
\newcommand{\mC}{\mathbb{C}}
\newcommand{\cK}{{\mathcal K}}
\newcommand{\D}{{\mathcal T}}

\date{\empty}
\begin{document}


\title{Null-controllability of linear hyperbolic systems in one dimensional space}
\author{Jean-Michel Coron\footnote{Sorbonne Universit\'{e}, Universit\'{e},
	Paris-Diderot SPC, CNRS, INRIA,
	Laboratoire Jacques-Louis Lions, \'{e}quipe Cage, Paris, France, coron\@ann.jussieu.fr} and Hoai-Minh Nguyen\footnote{Department of Mathematics, EPFL SB CAMA, 	Station 8 CH-1015 Lausanne, Switzerland, hoai-minh.nguyen\@epfl.ch}}

\maketitle

\begin{abstract} This paper is devoted to the controllability of a general linear hyperbolic system
 in one space dimension using boundary controls on one side. Under precise and generic assumptions on the boundary conditions on the other side, we previously  established the optimal time for the null and the exact controllability for this system for a generic source term. In this work, we prove the null-controllability for any time greater than the optimal time and for any source term.  Similar results for the exact controllability are also discussed.
\end{abstract}


\section{Introduction and statement of the main result}
Linear hyperbolic systems in one dimensional space are frequently used
in modeling of many systems such as traffic flow, heat exchangers,  and fluids in open channels. The
stability and boundary stabilization of these hyperbolic systems
have been studied intensively in the literature, see,  e.g.,   \cite{BastinCoron} and the references therein. In this paper, we are concerned about the optimal time for the null-controllability using boundary controls on one side.  More precisely, we consider the system
\begin{equation}\label{Sys-1}
\partial_t w (t, x) =  \Sigma(x) \partial_x w (t, x) + C(x) w(t, x)
\mbox{ for } (t, x)  \in \mR_+ \times (0, 1).
\end{equation}
Here $w = (w_1, \cdots, w_n)\tr : \mR_+ \times (0, 1) \to \mR^n$ ($n \ge 2$), $\Sigma$ and $C$ are   $(n \times n)$ real matrix-valued functions defined in $[0,1]$. As usual, see e.g.  \cite{CoronNg19}, we assume that, may be after a linear change of variables $w\rightarrow R(x)w$,  $\Sigma(x)$ is of the form
\begin{equation}\label{form-A}
\Sigma(x) = \mbox{diag} \big(- \lambda_1(x), \cdots, - \lambda_{k}(x), \lambda_{k+1}(x), \cdots,  \lambda_{n}(x) \big),
\end{equation}
where
\begin{equation}\label{relation-lambda}
-\lambda_{1}(x) < \cdots <  - \lambda_{k} (x)< 0  < \lambda_{k+1}(x) < \cdots < \lambda_{k+m}(x).
\end{equation}
Throughout the paper, we assume that
\begin{equation}\label{cond-lambda}
\mbox{$\lambda_i$  is Lipschitz on $[0, 1]$  for $1 \le i \le n \,  (= k + m)$}
\end{equation}
and
\begin{equation}\label{cond-C}
C \in [L^\infty(0, 1)]^{n \times n}.
\end{equation}
We are interested in the following type of boundary conditions and boundary controls. The boundary conditions at $x = 0$ is given by
\begin{equation}\label{bdry:x=0-w}
w_- (t, 0) = B w_+(t, 0)  \mbox{ for } t \ge 0,
\end{equation}
where $w_- = (w_1, \cdots, w_k)\tr $ and $w_+ = (w_{k+1}, \cdots, w_{k+m})\tr$,
for some given $(k \times m)$ real,  {\it constant} matrix $B$, and the boundary controls  at $x = 1$ is
\begin{equation}\label{control:x=1-w}
w_+ (t, 1) = W(t) \mbox{ for } t \ge 0,
\end{equation}
where $W = (W_{k +1}, \dots, W_{k + m})\tr$ are controls.

Let us recall that the control system \eqref{Sys-1}, \eqref{bdry:x=0-w}, and \eqref{control:x=1-w} is null-controllable (resp. exactly controllable) at the time $T>0$ if, for every initial data $w_0: (0,1)\to \mathbb{R}^n$ in $[L^2(0,1)]^n$ (resp. for every initial data $w_0: (0,1 )\to \mathbb{R}^n$ in $[L^2(0,1)]^n$  and for every (final) state $w_T:  (0,1 )\to \mathbb{R}^n$  in $[L^2(0,1)]^n$), there is a control $W :(0,T)\to \mathbb{R}^m$ in $[L^2(0,T)]^m$ such that the solution
of \eqref{Sys-1}, \eqref{bdry:x=0-w}, and \eqref{control:x=1-w} satisfying $w(0,x)=w_0(x)$ vanishes (resp. reaches $w_T$) at the time $T$: $w(T,x)=0$ (resp. $w(T, x)= w_T(x)$). Throughout this paper, we consider broad solutions in $L^2$ with respect to $t$ and $x$ for an initial data in $L^2(0, 1)$ as in \cite[Definition 3.1]{CoronNg19}. The well-posedness for broad solutions was given in \cite[Lemma 3.2]{CoronNg19}.  In fact, in \cite[Definition 3.1]{CoronNg19} and \cite[Lemma 3.2]{CoronNg19}, bounded broad solutions with respect to $t$ and $x$ for an initial data in $[L^\infty(0, 1)]^n$ are considered, nevertheless, the extension for $L^2$-setting is quite straightforward (see also \cite{Bressan00}).

\medskip
Set
\begin{equation}\label{def-tau}
\tau_i :=  \int_{0}^1 \frac{1}{\lambda_i(\xi)}  \, d \xi  \mbox{ for } 1 \le i \le n,
\end{equation}
and
\begin{equation}\label{def-Top}
T_{opt} := \left\{ \begin{array}{l}  \dsp \max \big\{ \tau_1 + \tau_{m+1}, \dots, \tau_k + \tau_{m+k}, \tau_{k+1} \big\}  \mbox{ if } m \ge k, \\[6pt]
\dsp \max \big\{ \tau_{k+1-m} + \tau_{k+1}, \tau_{k+2-m} + \tau_{k+2},  \dots,
\tau_{k} + \tau_{k+m} \big\}   \mbox{ if } m < k.
\end{array} \right.
\end{equation}

In this paper, we are mainly concerned about the optimal time for the null controllability of \eqref{Sys-1}, \eqref{bdry:x=0-w}, and \eqref{control:x=1-w} for $k \ge m \ge 1$. The null-controllability  was known from \cite{Russell78} for the time $\tau_k + \tau_{k+1}$ without any assumption on $B$ (see also \cite{AM16, CHO17, CVKB13, MVK13} for feedback controls using backstepping). In our previous work  \cite{CoronNg19}, we established the null controllability of \eqref{Sys-1}, \eqref{bdry:x=0-w}, and \eqref{control:x=1-w} at the {\it optimal} time $T_{opt}$ with $B \in \cB$ defined in \eqref{def-B} below, for a generic $C$, i.e. for $\gamma C$ with $\gamma \in \mR$ outside a discrete subset of $\gamma \in \mR$. When $C \equiv 0$, we also show that
there exists a {\it linear time independent feedback} which yields the null-controllability at the time $T_{opt}$.  Similar results for the exact controllability at $T_{opt}$ were also established there (see \Cref{sect-exact} for a discussion). The optimality of $T_{opt}$  even for $C \equiv 0$ was also discussed in \cite{CoronNg19}. It is worth noting that there are choices of  constant $\Sigma$, $B$, and $C$ when  $m=2$ and $k \ge 2$ so that the system is not null-controllable at the time $T_{opt}$ \cite[part 2 of Theorem 1]{CoronNg19} (see also \cite[pages 559-561]{Russell78}). It is easy to see  that $\cB$ is an open subset of  the set of (real) $k\times m$ matrices  and the Hausdorff dimension of its complement is $\min\{k, m-1 \}$.

\medskip

In this work, we prove the null-controllability of  \eqref{Sys-1}, \eqref{bdry:x=0-w}, and \eqref{control:x=1-w}  for any time greater than $T_{opt}$ and for $m \ge k \ge 1$ without the generic requirement.  Here is the main result of our paper:

\begin{theorem}\label{thm1} Let $k \ge m \ge 1$,  and set
\begin{equation}\label{def-B}
\cB: = \big\{B \in \mR^{k \times m}; \mbox{ such that  \eqref{cond-B-1} holds for  $1 \le i \le  \min\{k, m- 1\}$} \big\},
\end{equation}
where
\begin{equation}\label{cond-B-1}
\mbox{ the $i \times i$  matrix formed from the last $i$ columns and  rows of $B$  is invertible.}
\end{equation}
Assume  that  $B \in \cB$. The control system \eqref{Sys-1}, \eqref{bdry:x=0-w}, and \eqref{control:x=1-w} is null-controllable at any time $T$ greater than $T_{opt}$.
\end{theorem}

To our knowledge, the null-controllability result of \Cref{thm1} in the case $m < k$ with general $m$ and $k$ is new. The sharpest  known result on the time to obtain the null-controllability is $\tau_k + \tau_{k+1}$.   When $m = k$, \Cref{thm1} can be derived from the exact controllable result in \cite{HO19} under the additional assumption that \eqref{cond-B-1} holds for $i =k$ (see \Cref{sect-exact} for a discussion).
The starting point  of our analysis is the backstepping approach.
 More precisely, as in \cite{CoronNg19}, we make the following change of variables
$$
u(t, x) = w(t, x) - \int_0^x K(x, y) w(t, y) \, d y.
$$
Here the kernel  $K: \D = \big\{(x, y) \in (0, 1)^2; 0  <  y <  x \big\}
 \to \mR^{n}$ is chosen such that $u$ satisfies
 \begin{equation}\label{eq-u}
\partial_t u (t, x) =  \Sigma(x) \partial_x u (t, x) + S(x) u(t, 0)  \mbox{ for } (t, x)  \in  (0, T) \times (0, 1),
\end{equation}
where $S \in [L^\infty(0, 1)]^{n \times n}$ has the structure
\begin{equation}\label{cond-S}
S = \left( \begin{array}{cl}
0_{k, k} & S_{-+} \\[6pt]
0_{m, k} & S_{++}
\end{array}\right),
\end{equation}
with
$$
(S_{++})_{pq} = 0  \mbox{ for } 1 \le q \le p,
$$
$S_{-+} \in [L^\infty(0, 1)]^{k \times m}$ and $S_{++} \in [L^\infty(0, 1)]^{k \times k}$.  Here and in what follows, $0_{i, j}$  denotes the zero matrix  of size $i \times j$ for $i, \, j \in \mN$, and $M_{pq}$ denotes the $(p, q)$-component of a matrix $M$.
It is shown in \cite[Proposition 3.1]{CoronNg19} that  the null-controllability of \eqref{Sys-1}, \eqref{bdry:x=0-w}, and \eqref{control:x=1-w} at the time $T$ can be derived from the null-controllability at the time $T$ of \eqref{eq-u} equipped  the boundary condition at $x=0$
\begin{equation}\label{bdry-u}
u_- (t, 0)  = B u_+(t, 0) \mbox{ for } t \ge 0,
\end{equation}
and the boundary controls  at $x = 1$
\begin{equation}\label{control-u}
u_+ = U(t) \mbox{ for } t \ge 0 \mbox{ where $U$ is the control.}
\end{equation}
To establish the null-controllability for $u$, we use the Hilbert uniqueness method which involves crucially a compactness result type in  \Cref{lem-compact} with its roots in \cite{CoronNg19}.

The backstepping approach for the control of partial differential equations was pioneered by  Miroslav Krstic and his coauthors (see \cite{Krstic08} for a concise introduction).  The use of backstepping method to obtain the null-controllability for hyperbolic systems in one dimension  was initiated  in \cite{CVKB13} for the case $m = k = 1$. This approach has been  developed later on for more general  $m$ and $k$ in \cite{AM16, CHO17,MVK13}. The backstepping method is now frequently used for various control problems modeling by partial differential equations in one dimension. For example, it has been also used to stabilize the  wave equation  \cite{KGBS08, SK09, SCK10}, the parabolic equations in \cite{SK04,SK05},  nonlinear parabolic equations  \cite{Vazquez08},  and to  obtain the null-controllability of the heat equation \cite{Coron-Nguyen17}. The standard backstepping approach relies on the Volterra transform of the second kind. It is worth noting that, in some situations, more general transformations have to be  considered as for Korteweg-de Vries equations  \cite{CoronC13}, Kuramoto--Sivashinsky equations \cite{Coron15}, Schr\"{o}dinger's equation \cite{CoronGM18}, and hyperbolic equations with internal controls \cite{2019-Zhang-preprint}.

\medskip
The rest of the paper is organized as follows. In \Cref{sect-null}, we establish \Cref{thm1}. The exact controllability is discussed in \Cref{sect-exact}.

\section{Optimal time for the null-controllability}\label{sect-null}

In this section, we study the null-controllability of \eqref{eq-u} and \eqref{bdry-u} under the control law \eqref{control-u}. The main result of this section, which implies \Cref{thm1} by \cite[Proposition 3.5]{CoronNg19}, is:

\begin{theorem}\label{thm2} Let $k \ge m \ge 1$. System \eqref{eq-u} and \eqref{bdry-u} under the control law \eqref{control-u} is null-controllable at any time larger than $T_{opt}$.
\end{theorem}

The rest of this section contains two sections. In the first section, we present some lemmas used in the proof of \Cref{thm2}. The proof of \Cref{thm2} is given in the second section.

\subsection{Some useful lemmas} Fix $T>0$. Define
$$
\begin{array}{cccc}
\cF_T: &  [L^2(0, T)]^m &  \to & [L^2(0, 1)]^n \\[6pt]
& \cF_T (U) &  \mapsto & u(T, \cdot),
\end{array}
$$
where $u(\cdot, \cdot)$ is the solution of the system \eqref{eq-u}-\eqref{control-u} with $u(t = 0, \cdot) = 0$.

\begin{lemma}\label{lem-FT} We have, for $v \in [L^2(0, 1)]^n$,
\begin{equation*}
\cF_T^*(v) = \Sigma_+ (1) v_+(\cdot, 1) \mbox{ in } (0, T),
\end{equation*}
where $v(\cdot, \cdot)$ is the unique solution of the system
\begin{equation}\label{eq-v}
\partial_t v (t, x) =  \Sigma(x) \partial_x v (t, x) + \Sigma'(x)  v(t, x)
\mbox{ for } (t, x)  \in (-\infty, T) \times (0, 1),
\end{equation}
with, for $t < T$,
\begin{equation}\label{bdry-v0}
v_-(t, 1)  = 0,
\end{equation}
\begin{equation}\label{bdry-v}
\Sigma_+ (0) v_+(t, 0) = - B\tr \Sigma_- (0) v_- (t, 0) + \int_0^1  S_{-+}\tr(x) v_- (t, x) + S_{++}\tr(x) v_+(t, x) \, dx,
\end{equation}
and
\begin{equation}\label{initial-v}
v(t = T, \cdot) = v \mbox{ in } (0, 1).
\end{equation}
\end{lemma}

Throughout this paper, $\langle \cdot, \cdot \rangle$ denotes the Euclidean scalar product in the Euclidean space and $\langle \cdot, \cdot \rangle_{L^2(a, b)}$ denotes the scalar product in $L^2(a, b)$ for  $a < b$.

\begin{proof} We have
\begin{align*}
\langle U, \cF_T^*v \rangle_{L^2(0, T)}
 = &  \langle \cF_T U, v \rangle_{L^2(0, 1)} = \langle u(T, \cdot), v(T, \cdot) \rangle_{L^2(0, 1)} \\[6pt]
= &  \int_0^T \partial_t  \langle u(t, \cdot), v(t, \cdot) \rangle_{L^2(0, 1)} \, dt \\[6pt]
= &  \int_0^T  \langle \partial_t u(t, \cdot), v(t, \cdot) \rangle_{L^2(0, 1)} + \langle u(t, \cdot), \partial_t v(t, \cdot) \rangle_{L^2(0, 1)}  \, dt  \\[6pt]
= &  \int_0^T \int_0^1 \langle \Sigma(x) \partial_x u (t, x) + S(x)  u(t, 0) , v(t, x) \rangle  + \langle  u(t, \cdot), \partial_t v(t, \cdot) \rangle  \, d x \, dt \mbox{ by } \eqref{eq-u}.
\end{align*}
An integration by parts yields
\begin{multline*}
\int_0^T \int_0^1  \langle \Sigma(x) \partial_x u (t, x), v(t, x) \rangle \, dx \, dt
 = \int_0^T \int_0^1 - \langle \Sigma'(x) v(t, x) + \Sigma(x) \partial_x v (t, x), u(t, x) \rangle \, d t \\[6pt]
+ \int_0^T \langle u(t, 1), \Sigma (1) v(t, 1) \rangle - \int_0^T \langle u(t, 0), \Sigma (0) v(t, 0) \rangle \, dt.
\end{multline*}
Using the conditions on $u$ at $x = 0$ and $x =1$ (see \eqref{bdry-u} and \eqref{control-u}), and \eqref{bdry-v0}, we have
\begin{multline*}
\int_0^T \langle u(t, 1), \Sigma (1) v(t, 1) \rangle - \int_0^T \langle u(t, 0), \Sigma (0) v(t, 0) \rangle \, dt
= \int_0^T  \langle \Sigma_+ v_+, u_+ \rangle (t, 1) \, dt  \\[6pt]- \int_0^T \langle B\tr \Sigma_- v_- + \Sigma_+ v_+, u_+ \rangle (t,0) \, dt.
\end{multline*}
We then obtain
\begin{multline*}
\langle U, \cF_T^*v \rangle  =
 \int_0^T \int_0^1 \langle S(x)  u(t, 0) , v(t, x) \rangle
 + \int_0^T  \langle \Sigma_+ v_+, u_+ \rangle (t, 1) \, dt \\[6pt]
 - \int_0^T \langle B\tr \Sigma_- v_- + \Sigma_+ v_+, u_+ \rangle (t,0) \, dt.
\end{multline*}
Using the boundary condition \eqref{bdry-v}, we obtain
\begin{equation*}
\langle U, \cF_T^*v \rangle_{L^2(0, T)}  =  \int_0^T  \langle \Sigma_+ v_+, u_+ \rangle (t, 1) \, dt,
\end{equation*}
which  implies the conclusion.
\end{proof}

Similarly, we have the following result whose proof is omitted.

\begin{lemma}\label{lem-ST} Let $T>0$ and $u_0 \in [L^2(0, 1)]^n$. Assume that $u$ is the unique solution of \eqref{eq-u} and \eqref{bdry-u} with $u(t=0, \cdot) = u_0$ and $u_+(\cdot, 0) = 0$ for $t>0$. Then, for $v \in L^2(0, 1)$, we have
\begin{equation*}
\int_0^1 \langle u(T, x),  v(x) \rangle  \, dx  = \int_0^1 \langle u_0 (x), v(0, x ) \rangle  \, dx,
\end{equation*}
where $v(\cdot, \cdot)$ is the solution of \eqref{eq-v}-\eqref{initial-v}.
\end{lemma}

Combining \Cref{lem-FT} and \Cref{lem-ST},  making  a translation in time, and applying the Hilbert uniqueness method (see e.g. \cite[Chapter 2]{Coron09}), we obtain

\begin{lemma}\label{lem-observability} Let $T>0$. System \eqref{eq-u}-\eqref{control-u} is null controllable at the time $T$ if and only if, for some positive constant $C$,
\begin{equation}\label{observability}
\int_{-T}^0 |v_+(t, 1)|^2 \, dt \ge C \int_0^1 |v(-T, x)|^2 \, dx \; \forall \, v \in [L^2(0, 1)]^n,
\end{equation}
where $v(\cdot, \cdot)$ is the unique solution of the system
\begin{equation}\label{eq-v-O}
\partial_t v (t, x) =  \Sigma(x) \partial_x v (t, x) + \Sigma'(x)  v(t, x)   \mbox{ for } (t, x)  \in (-\infty, 0) \times (0, 1),
\end{equation}
with, $t < 0$,
\begin{equation}\label{bdry-v0-O}
v_-(t, 1)  = 0,
\end{equation}
\begin{equation}\label{bdry-v-O}
\Sigma_+ (0) v_+(t, 0) = - B\tr  \Sigma_- (0) v_- (t, 0)
+ \int_0^1  S_{-+}\tr(x) v_- (t, x) + S_{++}\tr(x) v_+(t, x) \, dx,
\end{equation}
and
\begin{equation}\label{initial-v-O}
v(t = 0, \cdot) = v \mbox{ in } (0, 1).
\end{equation}
\end{lemma}

Finally, we establish a compactness type result which is one of the key ingredients in the proof of \Cref{thm2}.

\begin{lemma}\label{lem-compact} Let $k \ge m \ge 1$, $B \in \cB$,  and $T \ge T_{opt}$. Assume that
$(v_N)$ be a sequence of solutions of \eqref{eq-v-O}-\eqref{bdry-v-O} $($with $v_N(0, \cdot)$ in $[L^2(0, 1)]^n$$)$ such that
\begin{gather}\label{thm2-generating-eq-a}
\sup_{N} \|v_N (-T, \cdot) \|_{L^2(0, 1)} < + \infty,
\\
\label{thm2-generating-eq-b}
\lim_{N \to + \infty} \|v_{N, +}(\cdot, 1) \|_{L^2(-T, 0)} = 0.
\end{gather}
We have, up to a subsequence,
\begin{equation}\label{thm2-claim1}
v_N(-T, \cdot) \mbox{ converges in } L^2(0, 1),
\end{equation}
and  the limit $V \in [L^2(0, 1)]^n$ satisfies the equation
\begin{equation}\label{thm2-claim2}
V =  {\mathcal K} V,
\end{equation}
for some compact operator $\mathcal{K}$ from  $[L^2(0, 1)]^n$ into itself. Moreover, $\cK$ depends only  on $\Sigma$, $S$, and $B$; in particular, $\cK$ is independent of $T$.
\end{lemma}

\begin{proof} Denote,  for $1 \le \ell \le m$,
$$
V_{N, \ell} = (v_{N, k - m  +  \ell + 1}, \cdots, v_{N, k})\tr,
$$
$$
W_{N, \ell} = (v_{N, k + 1}, \cdots, v_{N, k+\ell})\tr,
$$
and set, for $0 \le \ell \le m-1$,
$$
\rho_{\ell} = - (T - \tau_{ k  + \ell+1}),
$$
$$
\hcD_{\ell+1} = \Big\{ (t, s): t \in (\rho_{\ell + 1}, \rho_{\ell}); \rho_{m-1} \le s \le t \Big\},
$$
and
$$
\cD_{\ell+1} = \Big\{ (t, s): t \in (\rho_{\ell + 1}, \rho_{\ell}); t \le s \le \rho_0 \Big\},
$$
with the convention $\rho_{m} =  - T$.

Note that $v_{N, -}(\cdot, 0) = 0$ for $t \in (\rho_m, \rho_{m-1})$.

$\bullet$ We are first concerned about the time interval $(\rho_1, \rho_0)$ and $x=0$.   Using \eqref{thm2-generating-eq-b}
and  the characteristic method one gets that, for $k+2 \le j \le k+m$,
\begin{equation}\label{thm2-st1}
v_{N, j}(t, \cdot) \to 0 \mbox{ in } L^2(0, 1) \mbox{ for } t \in (\rho_{1}, \rho_{0})
\end{equation}
and
\begin{equation}\label{thm2-st2}
v_{N, j}(\cdot, 0) \to 0 \mbox{ in } L^2(\rho_{1}, \rho_{0}).
\end{equation}
Recall that, for $t \in (-T, 0)$,
\begin{equation}\label{thm2-S1}
\Sigma_+ (0) v_{N, +}(t, 0) = - B\tr  \Sigma_- (0) v_{N, -} (t, 0)  + \int_0^1 S_{-+}\tr(x) v_{N, -} (t, x) + S_{++}\tr(x) v_{N, +}(t, x) \, dx,
\end{equation}
and note, since $v_{N, -}(\cdot, 1) = 0$ in $(-T, 0)$,  that, in $(0, 1)$,
\begin{equation}\label{thm2-v1}
v_{N, j}(t, \cdot) = 0 \mbox{ for } t \in (\rho_{1}, \rho_{0}),  \mbox{ for } 1 \le j \le k-m + 1.
\end{equation}

First, consider the last $(m-1)$ equations of \eqref{thm2-S1} for $t \in (\rho_1, \rho_0)$.
Using \eqref{cond-B-1} with $i = m-1$, and \eqref{thm2-v1}, and viewing $v_{N, j}(\cdot, 0)$ in
$(\rho_{1}, \rho_{0})$ and $v_{N, j} (t, \cdot)$ in $(0, 1)$ for $k+2 \le j \le k+m$ as parameters,  we obtain
\begin{equation*}
V_{N, 1}(t, 0) = \int_{\rho_{m-1}}^t  G_{1}(t, s) V_{N, 1} (s, 0) \, ds  + \int_{t}^{\rho_{0}}H_1(t, s) W_{N, 1}(s) \, ds  + F_{N, 1}(t) \mbox{ in } (\rho_{1}, \rho_{0}),
\end{equation*}
for some $G_1 \in \big[L^\infty (\hcD_1)\big]^{(m-1) \times (m-1)}$ and $H_1 \in
\big[ L^\infty (\cD_1) \big]^{(m-1) \times 1}$
which depends only on $\Sigma$, $B$, and $S$, and for some $F_{N, 1} \in [L^2(\rho_{1}, \rho_{0})]^{m-1}$,which depends only on $\Sigma$, $B$, and $S$, and $v_{N, j}(\cdot, t)$  and  $v_{N, j}(t, \cdot)$ for $ t\in (\rho_1, \rho_0)$, and   for $k+2 \le j \le k+m$. Moreover, by \eqref{thm2-st1} and \eqref{thm2-st2},
$$
\mbox{$F_{N, 1} \to 0$ in $L^2(\rho_{1}, \rho_{0})$ as $N \to + \infty$}.
$$

Next consider the first equation of \eqref{thm2-S1} for $t \in (\rho_1, \rho_0)$.  Using the fact $(S_{++}\tr)_{1q} = 0$ for $1 \le q \le m$ by \eqref{cond-S}, and applying the characteristic method,  we have
$$
W_{N, 1} (t, 0) = Q_{1}  V_{N, 1} (t, 0) + \int_{\rho_{m-1}}^t  L_{1} (t, s)V_{N, 1} (s, 0) \, ds,
$$
for some constant $Q_1 \in \mR^{1 \times (m-1)}$,  and  for some $L_1 \in \big[L^\infty(\hcD_1) \big]^{1 \times (m-1)}$
both depending only on $\Sigma$, $B$, and $S$.

$\bullet$ Generally, let $1 \le \ell \le m$, and consider the time interval $(\rho_\ell, \rho_{\ell-1})$ and $x=0$.
As $N \to + \infty$, since,
$$
\| v_N(\cdot, 1)\|_{L^2(-T, 0)}  \to 0,
$$
it follows that, for $k+ \ell  + 1 \le j \le k+m$,
\begin{equation}\label{thm2-st3}
v_{N, j}(t, \cdot)  \to 0 \mbox{ in } L^2(0, 1) \mbox{ for } t \in (\rho_{\ell}, \rho_{\ell-1})
\end{equation}
and
\begin{equation}\label{thm2-st4}
v_{N, j}(\cdot, 0) \to 0 \mbox{ in } L^2(\rho_{\ell}, \rho_{\ell-1}).
\end{equation}
Note that, in $(0, 1)$,
\begin{equation}\label{thm2-v2}
v_{N, j}(t, \cdot) = 0 \mbox{ for } t \in (\rho_{\ell}, \rho_{\ell-1}) \mbox{ for } 1 \le j \le k-m + \ell.
\end{equation}

Consider the last $(m- \ell)$ equations of system \eqref{thm2-S1} for $t \in (\rho_\ell, \rho_{\ell-1})$.
Using \eqref{cond-B-1} with $i = m - \ell$, and \eqref{thm2-v2}, and  viewing $v_{N, j}(\cdot, 0)$ in $(\rho_{\ell}, \rho_{\ell -1})$ and $v_{N, j}(t, \cdot)$ in $(0, 1)$ for $k + \ell + 1 \le j \le k + m$ as parameters, we obtain, for $t \in (\rho_{\ell}, \rho_{\ell-1})$,
\begin{equation}\label{thm2-cp1}
V_{N, \ell}(t, 0) = \int_{\rho_{m-1}}^t  G_{\ell}(t, s) V_{N, \ell} (s, 0) \, ds + \int_{t}^{\rho_{0}} H_{\ell}(t, s) W_{N, \ell}(s) \, ds  + F_{N, \ell}(t),
\end{equation}
for some $G_{\ell} \in \big[L^\infty( \hcD_\ell) \big]^{(m - \ell) \times (m - \ell)}$
and $H_\ell \in \big[ L^\infty (\cD_\ell)\big]^{(m - \ell) \times \ell}$ 
which depends only on $\Sigma$, $B$, and $S$, and for some
$F_{N, \ell} \in [L^2(\rho_{\ell}, \rho_{\ell-1})]^{m- \ell}$
which depends only on $\Sigma$, $B$, and $S$, and $v_{N, j}(\cdot, 0)$ and $v_{N, j}(t, \cdot)$ for $t \in (\rho_\ell, \rho_{\ell-1})$, and for $k+ \ell +1  \le j \le k + m$. Moreover, by \eqref{thm2-st3} and \eqref{thm2-st4}, we have
$$
\mbox{$F_{N, \ell} \to 0$ in $L^2(\rho_{\ell}, \rho_{\ell-1})$ as $N \to + \infty$}.
$$

Next consider the first $\ell$ equations of \eqref{thm2-S1} for $t \in (\rho_\ell, \rho_{\ell-1})$.  We have
\begin{equation}\label{thm2-cp2}
W_{N, \ell} (t, 0) = Q_{\ell} V_{N, \ell} (t, 0)
+ \int_{\rho_{m-1}}^t  L_{\ell} (t, s)V_{N, \ell} (s, 0) \, ds  + \int_t^{\rho_{0}}
M_{\ell}(t, s) W_{N, \ell} (s, 0) \, ds.
\end{equation}
for some constant $Q_\ell \in \mR^{\ell \times (m-\ell)}$,  for some
$L_{\ell} \in
\big[ L^\infty (\hcD_\ell) \big]^{\ell \times (m - \ell)}$ and $M_\ell \in \big[L^\infty( \cD_\ell) \big]^{\ell \times \ell}$,
all  depending only on $\Sigma$, $B$, and $S$. In the case $\ell = m$, \eqref{thm2-cp1} is irrelevant and \eqref{thm2-cp2} is understood in the sense that the first two terms in the RHS are 0.

We have
\begin{itemize}
\item[i)] $v_{N, -}(-T, \cdot) = 0$ in $(0, 1)$;

\item[ii)] the information of $v_{N, +}(-T, \cdot )$ in $(0, 1)$ is encoded by the information of $v_{N, k+1}(\cdot, 0)$ on $(\rho_m, \rho_0)$, of $v_{N, k+2}(\cdot, 0)$ on $(\rho_m, \rho_1)$,  \dots, of $v_{N, k+m}(\cdot, 0)$ on $(\rho_m,  \rho_{m-1})$, by the characteristic method;

\item[iii)] Using \eqref{thm2-cp1} for $\ell = m-1$, one can solve $V_{N, m-1}$ as a function of $W_{N, m-1}$ and $F_{N, m-1}$. Continue the process with $\ell = m-2$, then with $\ell = m-3$, \dots, and finally with $\ell = 1$. Noting that
$$
v_{N, k - m + \ell + 1}(\cdot, 0) = 0 \mbox{ in $(\rho_{m-1}, \rho_\ell)$,}
$$
one can solve
$$
V_{N, 1}(\cdot, 0) \in L^2(\rho_{m-1}, \rho_0)  \times \cdots \times L^2(\rho_{m-1}, \rho_{m-2})
$$
as a function of $W_{N, 1} \in L^2(\rho_m, \rho_0) \times \cdots \times L^2(\rho_m, \rho_{m-1})$ and $F_{N, j}$ with $j = 1, \dots, m$,  and one has
$$
V_{N, 1}(\cdot, 0) = \cK_1 W_{N, 1}(\cdot, 0) + g_{N}.
$$
where $g_N  \in L^2(\rho_{m-1}, \rho_0) \times  \cdots \times L^2(\rho_{m-1}, \rho_{m-2})$ converges to 0 in the corresponding $L^2$-norm and $\cK_1$ is a compact operator depending only on $\Sigma$, $S$ and $B$.
\end{itemize}
The conclusion now follows from \eqref{thm2-cp2}. The proof is complete.
\end{proof}

\subsection{Proof of \Cref{thm2}}  The arguments are in the spirit of \cite{BLR92} (see also \cite{Rosier97}). For $T> T_{opt}$,  set
\begin{multline}\label{thm2-def-YT}
Y_T := \Big\{ V \in L^2(0, 1): V \mbox{ is the limit in $L^2(0, 1)$ of some subsequence of solutions $\big(v_N(-T, \cdot) \big)$ }\\[6pt]  \mbox{of  \eqref{eq-v-O}-\eqref{bdry-v-O} such that
 \eqref{thm2-generating-eq-a} and \eqref{thm2-generating-eq-b} hold}\Big\}.
\end{multline}
It is clear that $Y_T$ is a vectorial space. Moreover, by \eqref{thm2-claim2} and the compact property of $\cK$, we have
\begin{equation}\label{thm2-pro-YT-1}
\dim Y_T \le  C,
\end{equation}
for some positive constant $C$ independent of $T$.

We next show that
\begin{equation}\label{thm2-pro-YT-2}
Y_{T_2} \subset Y_{T_1} \mbox{ for }  T_{opt} < T_1 < T_2.
\end{equation}
Indeed,  let $V \in Y_{T_2}$. There exists a sequence of solutions $(v_N)$  of \eqref{eq-v-O}-\eqref{bdry-v-O} such that
\begin{equation}
\left\{\begin{array}{l}
v_N (-T, \cdot) \to V  \mbox{ in } L^2(0, 1),  \\[6pt]
\lim_{N \to + \infty} \|v_{N, +}(\cdot, 1) \|_{L^2(-T_2, 0)} = 0.
\end{array} \right.
\end{equation}
By considering the sequence $v_N(\cdot - \tau, \cdot)$ with $\tau = T_2 - T_1$, we derive that $V \in Y_{T_1}$.

 By \Cref{lem-observability}, to obtain the null-controllability at the time $T > T_{opt}$, it suffices to  prove \eqref{observability} by contradiction. Assume that  there exists a sequence of solutions $(v_N)$ of  \eqref{eq-v-O}-\eqref{bdry-v-O} such that
\begin{equation}\label{observability-contradiction}
N \int_{-T}^0 |v_{N, +}(t, 1)|^2 \, dt \le \int_0^1 |v_N(-T, x)|^2 \, dx =1.
\end{equation}
By \eqref{thm2-claim1}, up to a subsequence, $v_N(-T, \cdot)$ converges in $L^2(0, 1)$ to a limit $V$. It is clear that $\|V \|_{L^2(0, 1)} = 1$; in particular, $V \neq 0$. Consequently,
\begin{equation}\label{thm2-pro-YT-3}
Y_T \neq \{ 0\}.
\end{equation}

By \eqref{thm2-pro-YT-1}, \eqref{thm2-pro-YT-2},  and \eqref{thm2-pro-YT-3}, there exist $T_{opt} < T_{1} < T_2 < T$ such that
\begin{equation*}
\dim Y_{T_1} = \dim Y_{T_2} \neq 0.
\end{equation*}

We claim that, for $V \in Y_{T_1}$,
\begin{equation}\label{thm2-claimB}
 \Sigma \partial_x V + \Sigma' V  \mbox{ is an element in } Y_{T_1}.
\end{equation}
Indeed, since $Y_{T_1} = Y_{T_2}$, by the definition of $Y_{T_2}$, there exists a sequence of solutions $(v_N)$ of  \eqref{eq-v-O}-\eqref{bdry-v-O} such that
\begin{equation}\label{thm2-generating-eq-1}
\left\{\begin{array}{c}
\lim_{N \to + \infty} \|v_{N, +}(\cdot, 1) \|_{L^2(-T, 0)} = 0, \\[6pt]
V = \lim_{N \to + \infty} v_N(-T_2, \cdot) \mbox{ in } L^2(0, 1).
\end{array}\right.
\end{equation}
Using \eqref{thm2-pro-YT-2}, one may assume that $T_2 - T_1$ is small. We claim that, for $t \in(-T_2, T_1]$,
\begin{equation}\label{thm2-claimA}
\sup_n \|v_N(-t, \cdot) \|_{L^2(0, 1)} < +\infty.
\end{equation}
Noting that $\Sigma$ and $\Sigma'$ are diagonal, we have, by the characteristic method, for $t \in (-T_2, -T_{opt})$
\begin{equation}\label{thm2-claimA-p1}
v_{N, -} (t, \cdot)= 0 \mbox{ in } (0, 1).
\end{equation}
Using the characteristic method again, we also have, for $t \in (-T_2, -T_1]$,
\begin{equation}
\|v_{N, +}(t, \cdot) \|_{L^2(0, 1)} \le C \Big( \|v_{N, +}(-T_2, \cdot) \|_{L^2(0, 1)} + \|v_{N, +}(\cdot, 1) \|_{L^2(-T_2, t)} \Big).
\end{equation}
We derive from \eqref{thm2-generating-eq-1} that
\begin{equation}\label{thm2-claimA-p2}
\sup_n \|v_{N, +}(t, \cdot) \|_{L^2(0, 1)} < + \infty.
\end{equation}
Combining \eqref{thm2-claimA-p1} and \eqref{thm2-claimA-p2} yields \eqref{thm2-claimA}.

Using \eqref{thm2-claim1}, without loss of generality, one may assume that
$$
v_N(-T_1, \cdot ) \to \hat V \mbox{ in } L^2(0, 1) \mbox{ for some $\hat V \in L^2(0, 1)$}.
$$
Let $\hv$ be the unique solution of the system
\begin{equation}\label{thm2-eq-v}
\partial_t \hv (t, x) =  \Sigma(x) \partial_x \hv (t, x) + \Sigma'(x)  \hv(t, x)  \mbox{ for } (t, x)  \in (-\infty, -T_1) \times (0, 1),
\end{equation}
with, for $t < -T_1$,
\begin{equation}\label{thm2-bdry-v0}
v_-(t, 1)  = 0,
\end{equation}
\begin{equation}\label{thm2-bdry-v}
\Sigma_+ (0) \hv_+(t, 0) = - B\tr  \Sigma_- (0) \hv_- (t, 0)
+ \int_0^1  S_{-+}\tr(x) \hv_- (t, x) + S_{++}\tr(x) \hv_+(t, x) \, dx,
\end{equation}
and
\begin{equation}\label{thm2-initial-v}
\hv(t = -T_1, \cdot) = \hat V.
\end{equation}
One then has, for $\tau < -T_1$,
\begin{equation}
v_{N} \to \hv \mbox{ in } C^0\big([-\tau, -T_1]; L^2(0, 1) \big).
\end{equation}
In particular, by \eqref{thm2-pro-YT-2}, we have
\begin{equation}\label{thm2-hv1}
\hv(t, \cdot) \in Y_{T_1} \mbox{ for } t \in [-T_2, -T_1)
\end{equation}
and
$$
V = v(-T_2, \cdot) \mbox{ in } (0, 1).
$$
Since, in the distributional sense and hence in $Y_{T_1}$,
$$
\partial_t \hv(-T_2, \cdot) = \lim_{\eps  \to 0_+} \frac{1}{\eps} \Big[ \hv(-T_2 + \eps, \cdot) - \hv(-T_2, \cdot) \Big],
$$
and,  for $\eps > 0$ small,
$$
\frac{1}{\eps} \Big[\hv(-T_2 + \eps, x) - \hv(-T_2, x) \Big] \in Y_{T_1} \mbox{ by } \eqref{thm2-hv1},
$$
one derives that
$$
\Sigma \partial_x \hv(-T_2,  \cdot ) + \Sigma' \hv(-T_2, \cdot) \in Y_{T_1},
$$
which implies \eqref{thm2-claimB}.

Recall that $Y_{T_1}$ is real and  of finite dimension. Consider its natural extension as a complex vectorial space and still denote its extension by $Y_{T_1}$. Define
\begin{equation*}
\begin{array}{cccc}
\cA: & Y_{T_1} & \to  & Y_{T_1} \\[6pt]
& V & \mapsto & \Sigma \partial_x V + \Sigma' V.
\end{array}
\end{equation*}
From the definition of $Y_{T_1}$, it is clear that, for $V \in Y_{T_1}$,
\begin{equation}\label{thm2-bdry-V0}
V_-(1)  = 0
\end{equation}
and
\begin{equation}\label{thm2-bdry-V}
\Sigma_+ (0) V_+(0) = - B\tr  \Sigma_- (0) V_- (0)
+ \int_0^1  S_{-+}\tr(x) V_- (x) + S_{++}\tr(x) V_+(x) \, dx.
\end{equation}
Since $Y_{T_1} \neq \{0 \}$ and $Y_{T_1}$ is of finite dimension, there exists $\lambda \in \mC$ and $V \in Y_{T_1} \setminus \{0 \}$ such that
\begin{equation*}
\cA V = \lambda V.
\end{equation*}
Set
$$
v(t, x) = e^{\lambda t} V(x) \mbox{ in } (-\infty, 0) \times (0, 1).
$$
Using \eqref{thm2-bdry-V0} and \eqref{thm2-bdry-V}, one can verify that $v(t, x)$ satisfies \eqref{eq-v-O}-\eqref{bdry-v-O}. Using \eqref{thm2-bdry-V0} and applying the characteristic method, one deduce that
\begin{equation}\label{thm2-A1}
v_- (t, \cdot) = 0 \mbox{ for } t < 0.
\end{equation}
From \eqref{bdry-v-O},  we then obtain
\begin{equation}\label{thm2-bdry-V-new}
\Sigma_+ (0) v_+(t, 0) =  \int_0^1 S_{++}\tr(x) v_+(t, x) \, dx.
\end{equation}
Using the structure of $S_{++}$, we then have
$$
v_{k+1}(t, 0) = 0 \mbox{ for } t < 0.
$$
By the characteristic method, this in turn implies that, for $t < - \tau_{k+1}$.
$$
v_{k+1}(t, \cdot) = 0 \mbox{ in } (0, 1)
$$
Similarly, we have, for $t < - \tau_{k+1} - \tau_{k+2} $,
$$
v_{k+2}(t, \cdot) = 0 \mbox{ in } (0, 1)
$$
\dots, and
for $t < - \tau_{k+1} - \cdots - \tau_{k+m}$,
$$
v_{k+m}(t, \cdot) = 0 \mbox{ in } (0, 1).
$$
Then $v(t, \cdot ) =  0 $ in $(0, 1)$ for $t < - \tau_{k+1} - \cdots - \tau_{k+m} $. It follows that $V = 0$ which contradicts the fact $V \neq 0$.
Thus \eqref{observability} holds and the null-controllability  is valid for $T > T_{opt}$. \qed

\section{Optimal time for the exact controllability}\label{sect-exact}

This section is on the exact controllability of \eqref{Sys-1}, \eqref{bdry:x=0-w}, and \eqref{control:x=1-w} for $m \ge k \ge 1$. We give a new short proof,  in the spirit of the one of \Cref{thm1},  of the following result due to  Hu and Olive \cite{HO19}.

\begin{theorem}\label{thm3}
Assume that $m \ge k \ge 1$. Set
\begin{equation*}
\cB_e: = \big\{B \in \mR^{k \times m}; \mbox{ such that  \eqref{cond-B-1} holds for  $1 \le i \le  k$} \big\},
\end{equation*}
Assume that $B \in \cB_e$. The control system \eqref{Sys-1}, \eqref{bdry:x=0-w}, and \eqref{control:x=1-w} is exactly controllable at any time $T$ greater than $T_{opt}$.
\end{theorem}

The exact controllability of \eqref{Sys-1}, \eqref{bdry:x=0-w}, and \eqref{control:x=1-w}  for $m \ge k$ has been investigated intensively in the literature. When $m = k$  under  a similar condition, the exact controllability was considered in \cite[Theorem 3.2]{Russell78}. In the quasilinear case with $m \ge k$,  the exact controllability was derived in   \cite[Theorem 3.2]{Li00} (see also \cite{LiRao02}) for $m \ge k$ and  for the time $\tau_k + \tau_{k+1}$ under a condition which is equivalent to the fact that  \eqref{cond-B-1} holds for $1 \le i \le k$. The result was improved when $C = 0$ in \cite{Hu15} when the time of control is $\max\{\tau_{k+1}, \tau_{k} + \tau_{m+1} \}$ involving backstepping. The exact controllablility of \eqref{Sys-1}, \eqref{bdry:x=0-w}, and \eqref{control:x=1-w} at the time  $T_{opt}$ was recently established in \cite{CoronNg19} for a generic $C$, i.e., for $\gamma C$ with $\gamma \in \mR$ outside a discrete subset of  $\gamma \in \mR$ using the backstepping approach. The generic condition of $C$ is not required for $C$ with small $L^\infty$-norm by the same approach.  It is worth noting that $\cB_e$ is an open subset of  the set of (real) $k\times m$ matrices  and the Hausdorff dimension of its complement is $k$.
The generic condition  is then removed recently in \cite{HO19} by  a different approach.

\medskip

In this section, we show how to adapt the approach for \Cref{thm1} to derive \Cref{thm3}.  As in the study of the  null-controllability, it suffices, by  \cite[Proposition 3.1]{CoronNg19},  to establish

\begin{theorem}\label{thm4} Let $m \ge k \ge 1$. System \eqref{eq-u}-\eqref{bdry-u} under the control law \eqref{control-u} is exactly controllable at any time larger than $T_{opt}$.
\end{theorem}

As a consequence of \Cref{lem-FT}, by the Hilbert uniqueness principle, see, e.g., \cite[Chapter 2]{Coron09}, we have

\begin{lemma}\label{lem-observability-e} Let $T>0$. System \eqref{eq-u}-\eqref{control-u} is exactly controllable at the time $T$ if and only if, for some positive constant $C$,
\begin{equation}\label{observability-e}
\int_{-T}^0 |v_+(t, 1)|^2 \, dt \ge C \int_0^1 |v(0, x)|^2 \, dx \;  \forall \,   v \in [L^2(0, 1)]^n,
\end{equation}
for all solution $v(\cdot, \cdot)$ of \eqref{eq-v-O}-\eqref{bdry-v-O}.
\end{lemma}

As a variant of \Cref{lem-compact}, we establish
\begin{lemma}\label{lem-compact-e} Let $m \ge k \ge 1$, $B \in \cB_e$,  and $T \ge T_{opt}$. Assume that
$(v_N)$ be a sequence of solutions of \eqref{eq-v-O}-\eqref{bdry-v-O} such that
\begin{equation}\label{thm4-generating-eq}
\sup_{N} \|v_N (0, \cdot) \|_{L^2(0, 1)} < + \infty, \mbox{ and }
 \lim_{N \to + \infty} \|v_{N, +}(\cdot, 1) \|_{L^2(0, T)} = 0.
\end{equation}
We have, up to a subsequence,
\begin{equation}\label{thm4-claim1}
v_N(0, \cdot) \mbox{ converges in } L^2(0, 1),
\end{equation}
and the limit $V \in [L^2(0, 1)]^n$ satisfies the equation
\begin{equation}\label{thm4-claim2}
V =  {\mathcal K}_e V,
\end{equation}
for some compact operator $\mathcal{K}_e$ from  $[L^2(0, 1)]^n$ into itself. Moreover, $\cK_e$ depends only  on $\Sigma$, $S$, and $B$; in particular, $\cK_e$ is independent of $T$.
\end{lemma}

\begin{proof} Since  $\lim_{N \to + \infty} \|v_{N, +}(\cdot, 1) \|_{L^2(0, T)} = 0$ and $T \ge T_{opt}$, it follows from the characteristic method that, for $t \in (-T + \tau_{k+1}, 0]$,
\begin{equation}\label{thm4-p1}
\| v_{N, +} (t, \cdot)\|_{L^2(0, 1)} \to 0 \mbox{ as } N \to + \infty.
\end{equation}
By the characteristic method, we also have, for $1 \le j  \le k$,
\begin{equation}\label{thm4-p2}
\| v_{N, j} (t, \cdot)\|_{L^2(0, 1)} = 0 \mbox{ for } t \in (-T, -\tau_j).
\end{equation}

Recall that, for $t \le 0$,
\begin{equation}\label{thm4-bdry}
\Sigma_+ (0) v_{N, +}(t, 0) = - B\tr \Sigma_- (0) v_{N, -} (t, 0)  + \int_0^1  S_{-+}\tr(x) v_{N, -} (t, x) + S_{++}\tr(x) v_{N, +}(t, x) \, dx.
\end{equation}
Denote, for $1 \le j \le k$,
$$
V^e_{N, j} = (v_{N, j}, \cdots, v_{N, k})\tr,
$$
$$W^e_{N, j} = (v_{N, k+1}, \cdots, v_{N, m+ j -1})\tr,
$$
and set, for $1 \le j \le k$,
$$
 \hcD_{j}^e:= \Big\{ (t, s): t \in (-\tau_j, -\tau_{j-1});  t \le s \le 0 \Big\},
$$
and
$$
 \cD_{j}^e:= \Big\{ (t, s): t \in (-\tau_j, -\tau_{j-1});  -\tau_k \le s \le t \Big\},
$$
with the convention $\tau_{0} = 0$. When $m = k$ and $j=1$, $W_{N, 1}^e$ is irrelevant.

For $1 \le j \le k$,  consider $t \in (-\tau_j, - \tau_{j-1})$ and $x = 0$. First, consider the last $(k- j + 1)$ equations of  \eqref{thm4-bdry} for $t \in (-\tau_j, - \tau_{j-1})$.  Using \eqref{thm4-p2} and  \eqref{cond-B-1} with $i= k - j + 1$, and viewing $v_{N, \ell}(t, \cdot)$ for $x \in (0, 1)$ and $v_{N, \ell}(\cdot, 0)$ for $t \in (-\tau_{j},  - \tau_{j-1})$  for  $m + j \le \ell \le k + m$ as parameters,  we have, for $t \in (-\tau_j, -\tau_{j-1})$,
\begin{equation}\label{thm4-p3}
V^e_{N, j} (t, 0) = \int_{-\tau_k}^{t} G_j^e (t, s) V^e_{N, j} (s, 0) \, ds
+  \int_{t}^0 H_j^e(t, s) W^e_{N, j} (s, 0) \, ds + F^e_{N, j}(t),
\end{equation}
for some $G_j^e \in [L^\infty(\hcD_j^e)]^{(k-j+1) \times (k- j + 1)}$  and  $H_j^e \in [L^\infty(\cD_j^e)]^{(k-j+1) \times (m-k+j-1)}$ and
which depends only on $\Sigma$, $S$, and $B$, and for some
$F^e_{N, j}\in [L^2(-\tau_j, -\tau_{j-1})]^{k-j+1}$, which depends only on $\Sigma$, $S$, and $B$, and $v_{N, +}$. Moreover, by \eqref{thm4-p1} and \eqref{thm4-p2},
\begin{equation}\label{thm4-p4}
F^e_{N, j} \to 0 \mbox{ in } L^2(-\tau_j, -\tau_{j-1}) \mbox{ as } N \to + \infty.
\end{equation}
When $k = m$ and $j=1$, the second term in the RHS of \eqref{thm4-p3} is understood by $0$.

Next, consider the first $(m-k + j -1)$ equations of  \eqref{thm4-bdry} for $t \in (-\tau_j, \tau_{j-1})$. We have
\begin{equation}\label{thm4-p4-new}
W_{N, j}^e (t, 0) = Q_{j}^e V_{N, j}^e (t, 0)
+ \int_{-\tau_k}^t  L_{j}^e (t, s) V^e_{N, j} (s, 0) \, ds   + \int_t^{0}
M_{\ell}^e (t, s) W_{N, j}^e (s, 0) \, ds.
\end{equation}
for some constant $Q_j^e  \in \mR^{(m-k + j -1) \times (k-j+1)}$,  for some  $ L_{j}^e \in [L^\infty( \cD_j^e)]^{(m-k + j -1) \times (k - j +1)}$, and for some $M_j^e \in  [L^\infty(\hcD_j)]^{(m-k + j - 1) \times (m-k+ j -1)}$, all  depending only on $\Sigma$, $B$, and $S$. When $k = m$ and $j=1$,  \eqref{thm4-p4-new} is irrelevant.

Using \eqref{thm4-p4} with $j=1$, one can solve $W_{N, 1}^e$ as a function of $V_{N, 1}^e$ and $F_{N, 1}^e$ (if $m = k$, then this  is irrelevant). Continue the process with $j=2$, then $j=3$, \dots, finally with $j = k$. Noting that
$$
v_{N, m+ j -1}(\cdot, 0) \to 0 \mbox{ in } L^2(\tau_{j-1}, 0),
$$
considering it as a parameters, and using \eqref{thm4-p4},  one can solve
\begin{equation*}
W_{N, k}^e \in [L^2(-\tau_k, 0)]^{m-k} \times L^2(-\tau_k, - \tau_1)
\times \dots \times L^2(-\tau_k, -\tau_{k-1})
\end{equation*}
as a function of $V_{N, k}^e \in L^2(-\tau_1, 0) \times \dots \times L^2(-\tau_k, 0)$,
and $F_{N, j}^e$ with $j = 1, \dots, k$,  and one has
$$
V_{N, k}^e = \cK_1^e W_{N, k}^e + g_{N}^e.
$$
where $g_N^e  \in L^2(-\tau_1, 0) \times \dots \times L^2(-\tau_k, 0)$ converges to 0 in the corresponding $L^2$-norm and $\cK_1^e$ is a compact operator depending only on $\Sigma$, $S$ and $B$. The conclusion now follows from \eqref{thm4-p3} after noting that the information of $v_{N, -}(0, \cdot )$ is encoded by the information of $v_{N, 1}(\cdot, 0)$ on $(-\tau_1, 0)$, of $v_{N, 2}(\cdot, 0)$ on $(-\tau_2, 0)$,  \dots, of $v_{N, k}(\cdot, 0)$ on $(-\tau_k,  0)$, by the characteristic method.
\end{proof}

We are ready to give the

\begin{proof}[Proof of \Cref{thm4}]  The proof of \Cref{thm4} is similar to the one \Cref{thm3}.  For $T> T_{opt}$, set
\begin{multline}\label{thm4-def-YT}
Y_T^e := \Big\{ V \in L^2(0, 1): V \mbox{ is the limit in $L^2(0, 1)$ of some subsequence of solutions $\big(v_N(0, \cdot) \big)$} \\[6pt]
\mbox{of  \eqref{eq-v-O}-\eqref{bdry-v-O} such that
 \eqref{thm4-generating-eq} holds}\Big\}.
\end{multline}
As in Theorem~\ref{thm2},  $Y^e_T$ is a vectorial space of finite dimension and  there exist $T_{opt} < T_1 < T_2 < T$ such that
$$
\dim Y_{T_1}^e = \dim Y_{T_2}^e.
$$
Fix such $T_1$ and $T_2$. By \Cref{lem-observability-e}, it suffices to  prove \eqref{observability-e} by contradiction. Assume that \eqref{observability-e} does not hold. Then, as in the proof Theorem~\ref{thm2},   there exist $\lambda \in \mC$ and $V \in Y_{T_1}^e \setminus \{0 \}$ such that
\begin{equation*}
 \Sigma \partial_x V + \Sigma' V = \lambda V.
\end{equation*}
Set
\begin{equation}\label{def-v-Thm4}
v(t, x) = e^{\lambda t} V(x) \mbox{ in } (-\infty, 0) \times (0, 1).
\end{equation}
As in the proof of \Cref{thm2}, one can verify that $v(\cdot, \cdot)$ satisfies \eqref{eq-v-O}-\eqref{bdry-v-O}.  Applying the characteristic method, one deduce that
\begin{equation}\label{thm2-A1-new}
v_- (t, \cdot) = 0 \mbox{ for } t < -\tau_k.
\end{equation}
As in the proof of \Cref{thm2}, we also have
\begin{equation}\label{thm4-cl}
\mbox{$v(t, \cdot ) =  0 $ in $(0, 1)$ for $t < - \tau_k - \tau_{k+1} - \cdots - \tau_{k+m} $}.
\end{equation}
It follows that $V = 0$ which contradicts the fact $V \neq 0$.
Thus \eqref{observability-e} holds and the exact-controllability  is valid for $T > T_{opt}$.
\end{proof}

\begin{remark} \rm \Cref{thm3} can be also deduced from \Cref{thm1}. Indeed, consider first the case $m  = k$. By making a change of variables
$$
\tw (t, x) = w(T -t, x) \mbox{ for } t \in (0, T), \, x \in (0, 1).
$$
Then
$$
\tw_-(t, 0) = \widetilde B^{-1} \tw_{+}(t, 0),
$$
with $\tw_-(t, \cdot)= (w_{2k}, \dots, w_{k+1})\tr (T-t, \cdot)$, and  $\tw_+ (t, \cdot) = (w_k, \dots, w_1)\tr (T -t, \cdot)$, and $\widetilde B_{ij} = B_{pq}$ with $p = k-i$ and $q = k-j$. Note that the $i \times i$  matrix formed from the first $i$ columns and  rows of $\widetilde B$  is invertible.
Using Gaussian elimination method, one can find  $(k \times k)$ matrices  $T_1, \dots, T_N$ such that
$$
T_N \dots T_1 \widetilde B = U,
$$
where $U$ is a $(k \times k)$ upper triangular matrix, and $T_i$ ($1 \le i \le N$) is the matrix given by  the operation which replaces a row $p$ by itself plus a multiple of a row $q$ for some $1\le q<p \le N$. It follows that
$$
\widetilde B^{-1} = U^{-1} T_N \dots T_1.
$$
One can check that $U^{-1}$ is an invertible,  upper triangular matrix and $T_N \dots T_1$ is an invertible,  lower triangular matrix. It follows that the
$i \times i$  matrix formed from the last $i$ columns and  rows of $\widetilde B^{-1}$ is the product of the matrix formed from the last $i$ columns and  rows of $U^{-1}$
and the matrix formed from the last $i$ columns and  rows of $T_N \dots T_1$. Therefore,  $\widetilde B^{-1} \in {\cal B}$. One can also check  that the exact controllability of the system for $w(\cdot, \cdot)$ at the time $T$ is equivalent to the null-controllability  of the system for $\tw(\cdot, \cdot)$ at the same time and the conclusion of \Cref{thm3} follows from \Cref{thm1}. The case $m>k$ can be obtained from the case $m = k$ as follows. Consider $\hw (\cdot, \cdot)$ the solution of the system
$$
\partial_t \hw = \hat \Sigma(x) \partial_x \hw (t, x) + \hat C(x) \hw (t, x),
$$
$$
\hw_-(t, 0) = \hat B \hw_+ (t, 0),  \quad \mbox{ and } \quad \hw_+(t, 1) \mbox{ are controls}.
$$
Here
$$
\hat \Sigma = \mbox{diag} (- \hat \lambda_1, \dots, -\hat \lambda_m, \hat \lambda_{m+1}, \dots \hat \lambda_{2m}),
$$
with $\hat \lambda_j =  - (1 + m -k - j ) \eps^{-1}$  for $1 \le j \le m-k$ with positive small $\eps$, $\hat \lambda_{j} = \lambda_{j - (m-k)}$ if $ m-k + 1 \le j \le m$, and $\hat \lambda_{j+m} = \lambda_{j + k}$ for $1 \le j \le m$,
$$
\hat C(x) = \left(\begin{array}{cc} 0_{m-k, m-k} & 0_{m-k, n}\\[6pt]
0_{n, m-k} & C (x)
\end{array}\right),
$$
and
$$
\hat B = \left(\begin{array}{cc} I_{m-k} & 0_{m-k, m}\\[6pt]
0_{m-k, m} & B
\end{array}\right),
$$
where $I_{\ell}$ denotes the identity matrix of size $\ell \times \ell$ for $\ell \ge 1$. Recall that $0_{i, j}$ denotes the zero matrix of size $i \times j$ for $i, j, \ell \ge 1$. Then the exact controllability of $w$ at the time $T$ can be derived from the exact controllability of $\hw$ at the same time. One then can deduce the conclusion of \Cref{thm3} from the case $m=k$ using \Cref{thm1} by noting that the optimal time for the system of $\hw$ converges to the optimal time for the system of  $w$ as $\eps \to 0_+$.
\end{remark}

\textbf{Acknowledgments.} The authors are partially supported by  ANR Finite4SoS ANR-15-CE23-0007.

\providecommand{\bysame}{\leavevmode\hbox to3em{\hrulefill}\thinspace}
\providecommand{\MR}{\relax\ifhmode\unskip\space\fi MR }
\providecommand{\MRhref}[2]{%
  \href{http://www.ams.org/mathscinet-getitem?mr=#1}{#2}
}
\providecommand{\href}[2]{#2}

\end{document}